%% file: non-linear-connection.tex
\documentclass[a4paper]{amsproc}

\usepackage[foot]{amsaddr}
\usepackage[utf8]{inputenc}
\usepackage[english]{babel}
\usepackage{amsmath}
\usepackage{amsfonts}
\usepackage{amssymb}
\usepackage{amsthm}
\usepackage{tikz}
\usepackage{mathtools}
\mathtoolsset{showonlyrefs,showmanualtags} 

\newtheorem{theorem}{Theorem}[section]
\newtheorem*{theorem*}{Theorem}

\newtheorem{lemma}[theorem]{Lemma}
\newtheorem*{lemma*}{Lemma}
\newtheorem{proposition}[theorem]{Proposition}
\newtheorem*{proposition*}{Proposition}

\theoremstyle{definition} 
\newtheorem{definition}[theorem]{Definition}
\newtheorem*{definition*}{Definition}

\theoremstyle{remark}
\newtheorem{remark}[theorem]{Remark}

\newcommand{\bi}{\begin{itemize}}

\newcommand{\ei}{\end{itemize}}
\newcommand{\bd}{\begin{description}}
\newcommand{\ed}{\end{description}}

\newcommand{\bqn}{\begin{eqnarray}}
\newcommand{\eqn}{\end{eqnarray}}

\newcommand{\la}{\langle}
\newcommand{\ra}{\rangle}

\newcommand{\g}{\gamma}
\newcommand{\al}{\alpha}

\newcommand{\eps}{\varepsilon}			
\newcommand{\lam}{\lambda}

\newcommand{\wt}[1]{\widetilde{#1}}
\newcommand{\mc}[1]{\mathcal{#1}}
\newcommand{\distr}{\mathcal{D}}				
\newcommand{\ver}{\mathcal{V}}					
\newcommand{\hor}{\mathcal{H}}					
\newcommand{\J}{\mathcal{J}}					
\newcommand{\R}{\mathbb{R}}						
\newcommand{\DD}{\mathcal{F}}					
\newcommand{\tanf}{\mathsf{T}}					
\newcommand{\y}{D}								
\newcommand{\Rcan}{\mathfrak{R}}				

\DeclareMathOperator{\spn}{\mathrm{span}}		
\DeclareMathOperator{\rank}{\mathrm{rank}}		

\author{Davide Barilari$^\flat$}
\address{$^\flat$Institut de Math\'ematiques de Jussieu-Paris Rive Gauche UMR CNRS 7586, Universit\'e Paris-Diderot,
Batiment Sophie Germain, Case 7012, 75205 Paris Cedex 13, France} \email{davide.barilari@imj-prg.fr}
\author{Luca Rizzi$^\sharp$}
\address{$^\sharp$CMAP \'Ecole Polytechnique, Palaiseau and \'Equipe INRIA GECO Saclay \^Ile-de-France, Paris, France}
\email{luca.rizzi@cmap.polytechnique.fr}


\date{\today}

\title[Jacobi fields and a canonical connection in sR geometry]{On Jacobi fields and a canonical connection in sub-Riemannian geometry}

\begin{document}

\begin{abstract}
In sub-Riemannian geometry the coefficients of the Jacobi equation define curvature-like invariants. We show that these coefficients can be interpreted as the curvature of a canonical Ehresmann connection associated to the metric, first introduced in \cite{lizel}. We show why this connection is naturally nonlinear, and we discuss some of its properties.
\end{abstract}

\maketitle
\tableofcontents

\section{Introduction}

A key tool for comparison theorems in Riemannian geometry is the Jacobi equation, i.e.~the differential equation satisfied by Jacobi fields. Assume $\gamma_{\eps}$ is a one-parameter family of geodesics on a Riemannian manifold $(M,g)$ satisfying
\begin{equation}\label{eq:geo0}
\ddot{\gamma}_{\eps}^{k}+\Gamma_{ij}^{k}(\gamma_{\eps})\dot\g^{i}_{\eps}\dot\g^{j}_{\eps}=0.
\end{equation}
The corresponding Jacobi field $J=\left.\frac{\partial}{\partial \eps}\right|_{\eps=0} \gamma_{\eps}$ is a vector field  defined along $\gamma=\gamma_{0}$, and satisfies the equation
\begin{equation}\label{eq:jacobo0}
\ddot{J}^{k}+2\Gamma_{ij}^{k}\dot J^{i}\dot\g^{j}+\frac{\partial \Gamma_{ij}^{k}}{\partial x^{\ell}}J^{\ell} \dot\g^{i}\dot\g^{j}=0.
\end{equation}
The Riemannian curvature is hidden in the coefficients of this equation. To make it appear explicitly, however, one has to write \eqref{eq:jacobo0} in terms of a parallel transported frame $X_{1}(t),\ldots,X_{n}(t)$ along $\g(t)$. Letting $J(t)=\sum_{i=1}^{n} J_{i}(t)X_{i}(t)$ one gets the following normal form:
\begin{equation}\label{eq:jacobo}
\ddot{J}_{i}+R_{ij}(t) J_{j}=0.
\end{equation}
Indeed the coefficients $R_{ij}$ are related with the curvature $R^{\nabla}$ of the unique linear, torsion free and metric  connection $\nabla$ (Levi-Civita) as follows
\[
R_{ij}=g(R^{\nabla}(X_{i},\dot \g)\dot \g,X_{j}).
\]
Eq.~\eqref{eq:jacobo} is the starting point to prove many results in Riemannian geometry. In particular, bounds on the curvature (i.e.~on the coefficients $R$, or its trace) have deep consequences on the analysis and the geometry of the underlying manifold.

In the sub-Riemannian setting this construction cannot be directly generalized. Indeed, the analogous of the Jacobi equation is a first-order system on the cotangent bundle that cannot be written as a second-order equation on the manifold. Still one can put it in a normal form, analogous to \eqref{eq:jacobo}, and study its coefficients \cite{lizel}. These appear to be the correct objects to bound in order to control the behavior of the geodesic flow and get comparison-like results (see for instance \cite{lizel2,BR-comparison}). Nevertheless one can wonder if these coefficients can arise, as in the Riemannian case, as the curvature of a suitable connection. We answer to this question, by showing that these coefficients are part of the curvature of a nonlinear canonical Ehresmann connection associated with the sub-Riemannian structure. In the Riemannian case this reduces to the classical, linear, Levi-Civita connection.

\subsection{The general setting} 
A sub-Riemannian structure is a triple $(M,\distr,g)$ where $M$ is smooth $n$-dimensional manifold, $\distr$ is a smooth, completely non-integrable vector sub-bundle of $TM$ and $g$ is a smooth scalar product on $\distr$.  Riemannian structures are included in this definition, taking $\distr=TM$. The sub-Riemannian distance is the infimum of the length of absolutely continuous admissible curves joining two points. Here admissible means that the curve is almost everywhere tangent to the distribution $\distr$, in order to compute its length via the scalar product $g$. The totally non-holonomic assumption on $\distr$ implies, by the Rashevskii-Chow theorem, that the distance is finite on every connected component of $M$, and the metric topology coincides with the one of $M$.  A more detailed introduction on sub-Riemannian geometry can be found in \cite{montgomerybook,nostrolibro,noterifford,notejean}.

In Riemannian geometry, it is well-known that the geodesic flow can be seen as a Hamiltonian flow on the cotangent bundle $T^{*}M$, associated with the Hamiltonian 
\begin{equation}
H(p,x)=\frac{1}{2}\sum_{i=1}^{n} \la p, X_{i}(x)\ra^{2},\qquad (p,x)\in T^{*}M,
\end{equation}
where $X_{1},\ldots,X_{n}$ is any local orthonormal frame for the Riemannian structure{, and the notation $\langle p, v\rangle$ denotes the action of a covector $p \in T_x^*M$ on a vector $v \in T_x M$.}
In the sub-Riemannian case, the Hamiltonian is defined by the same formula, where the sum is taken over a local orthonormal frame $X_{1},\ldots,X_{k}$ for $\distr$, with $k=\rank \distr$. The restriction of $H$ to each fiber is a degenerate quadratic form, but Hamilton's equations are still defined. These can be written as a flow on $T^*M$
\begin{equation}
\dot\lambda = \vec{H}(\lam),\qquad \lam\in T^{*}M,
\end{equation}
where $\vec{H}$ is the Hamiltonian vector field associated with $H$. This system cannot be written as a second order equation on $M$ as in \eqref{eq:geo0}. The projection $\pi : T^*M \to M$ of its integral curves are geodesics, i.e.~locally minimizing curves. In the general case, some geodesics may not be recovered in this way.  These are the so-called strictly abnormal geodesics \cite{montgomeryabnormal}, and they are related with hard open problems in sub-Riemannian geometry \cite{agrachevopen}.

In what follows, with a slight abuse of notation, the term ``geodesic'' refers to the not strictly abnormal ones.

An integral line of the Hamiltonian vector field $\lambda(t)=e^{t\vec{H}}(\lam) \in T^{*}M$, with initial covector $\lam$ is called \emph{extremal}. Notice that the same geodesic may be the projection of two different extremals. For these reasons, it is convenient to see the Jacobi equation as a first order equation for vector fields on $T^*M$, associated with an extremal, rather then a second order system on $M$, associated with a geodesic.

\section{Jacobi equation revisited}
For any vector field $V(t)$ along an extremal $\lambda(t)$ of the sub-Riemannian Hamiltonian flow, a dot denotes the Lie derivative in the direction of $\vec{H}$:
\begin{equation}
\dot{V}(t) := \left.\frac{d}{d\eps}\right|_{\eps=0} e^{-\eps \vec{H}}_* V(t+\eps).
\end{equation}
A vector field $\J(t)$ along $\lam(t)$ is called a \emph{sub-Riemannian Jacobi field} if it satisfies 
\begin{equation}\label{eq:defJF}
\dot{\J} = 0.
\end{equation}
The space of solutions of \eqref{eq:defJF} is a $2n$-dimensional vector space. The projections $J=\pi_{*}\J$ are vector fields on $M$ corresponding to one-parameter variations of $\g(t)=\pi(\lam(t))$ through geodesics; in the Riemannian case, they coincide with the classical Jacobi fields.

We intend to write \eqref{eq:defJF} using the natural symplectic structure $\sigma$ of $T^{*}M$.  First, observe that on $T^*M$ there is a natural smooth sub-bundle of Lagrangian\footnote{A Lagrangian subspace $L \subset \Sigma$ of a symplectic vector space $(\Sigma,\sigma)$ is a subspace with $\dim L = \dim\Sigma/2$ and $\sigma|_{L} = 0$.} spaces:
\begin{equation}
\ver_{\lambda} := \ker \pi_*|_{\lambda} = T_\lambda(T^*_{\pi(\lambda)} M).
\end{equation}
We call this the \emph{vertical subspace}. Then, pick a Darboux frame $\{E_i(t),F_i(t)\}_{i=1}^{n}$ along $\lambda(t)$.  It is natural to assume that $E_1,\ldots,E_n$ belong to the vertical subspace. To fix the ideas, one can think at the canonical basis $\{\partial_{p_i}|_{\lambda(t)},\partial_{x_i}|_{\lambda(t)}\}$ induced by a choice of coordinates $(x_1,\ldots,x_n)$ on $M$. 

In terms of this frame, $\J(t)$ has components $(p(t),x(t)) \in \R^{2n}$:
\begin{equation}
\J(t) = \sum_{i=1}^n p_{i}(t) E_{i}(t) + x_{i}(t) F_{i}(t).
\end{equation}
The elements of the frame satisfy
\begin{equation}\label{eq:Jacobiframe}
\begin{pmatrix}
\dot{E} \\
\dot{F}
\end{pmatrix} = 
\begin{pmatrix}
C_1(t)^{*} & -C_2(t) \\
R(t) & -C_1(t)
\end{pmatrix} \begin{pmatrix}
E\\
F
\end{pmatrix},
\end{equation}
for some smooth families of $n\times n$ matrices $C_1(t),C_2(t),R(t)$, where $C_2(t) = C_2(t)^*$ and $R(t)= R(t)^*$. We  stress  that the particular structure of the equations is implied solely by the fact that the frame is Darboux, that is
\begin{equation}
\sigma(E_i,E_j) = \sigma(F_i,F_j) = \sigma(E_i,F_j) -\delta_{ij} = 0, \qquad i,j=1,\ldots,n.
\end{equation}
Moreover, $C_2(t) \geq 0$ as a consequence of the non-negativity of the sub-Riemannian Hamiltonian. To see this, for a bilinear form $B: V\times V \to \R$ and $n$-tuples $v,w \in V$ let $B(v,w)$ denote the matrix $B(v_i,w_j)$. With this notation
\begin{equation}
C_2(t) = \sigma(\dot{E},E)|_{\lambda(t)} = 2H(E,E)|_{\lambda(t)} \geq 0,
\end{equation}
where we identified $\ver_{\lambda(t)} \simeq T_{\gamma(t)}^*M$ and we see the Hamiltonian as a symmetric bilinear form on fibers. In the Riemannian case, $C_2(t) > 0$.  In turn, the Jacobi equation, written in terms of the components $(p(t),x(t))$, becomes
\begin{equation}
\begin{pmatrix}\label{eq:Jacobicoord}
\dot{p} \\ \dot{x}
\end{pmatrix} = \begin{pmatrix} - C_1(t) & -R(t) \\ C_2(t) & C_1(t)^{*}
\end{pmatrix} \begin{pmatrix}
p \\ x
\end{pmatrix}.
\end{equation}

\section{The Riemannian case}
In the Riemannian case one can choose a suitable frame to simplify~\eqref{eq:Jacobicoord} as much as possible. 
 Let $X_1,\ldots,X_n$ be a parallel transported frame along the geodesic $\gamma(t)$.  Let $h_i:T^*M \to \R$ be the fiber-wise linear functions, defined by $h_i(\lambda):= \langle \lambda, X_i\rangle$. Indeed $h_{1},\ldots,h_{n}$ define coordinates on each fiber, and the vectors $\partial_{h_{i}}$. We define a moving frame along the extremal $\lambda(t)$ as follows
\begin{equation}
E_i:=\partial_{h_i}, \qquad F_i := -\dot{E}_i.
\end{equation}
One can recover the original parallel transported frame by projection, namely $\pi_* F_i|_{\lambda(t)} = X_i|_{\g(t)}$. We state here the properties of the moving frame. 
\begin{proposition}\label{p:riemcan}
The smooth moving frame $\{E_i,F_i\}_{i=1}^n$ satisfies:
\begin{itemize}
\item[(i)] $\pi_{*}E_i|_{\lambda(t)}=0$.
\item[(ii)] It is a Darboux basis, namely
\[
\sigma(E_i,E_j) = \sigma(F_i,F_j) = \sigma(E_i,F_j) - \delta_{ij} = 0, \qquad i,j=1,\ldots,n.
\]
\item[(iii)] The frame satisfies the structural equations
\[
\dot{E}_i = - F_i, \qquad \dot{F}_i = \sum_{j=1}^n R_{ij}(t) E_j,
\]
for some smooth family of $n\times n$ symmetric matrices $R(t)$.
\end{itemize}
If $\{\wt{E}_i,\wt{F}_j\}_{i=1}^n$ is another smooth moving frame along $\lambda(t)$ satisfying (i)-(iii), for some matrix $\wt{R}(t)$ then there exist a constant, orthogonal matrix $O$ such that 
\begin{equation}\label{eq:orthonormal}
\wt{E}_i|_{\lambda(t)} = \sum_{j=1}^n O_{ij}E_j|_{\lambda(t)}, \qquad  \wt{F}_i|_{\lambda(t)} = \sum_{j=1}^nO_{ij}F_j|_{\lambda(t)}, \qquad \wt{R}(t) = O R(t) O^*. 
\end{equation}
\end{proposition}
Thanks to this proposition, the symmetric matrix $R(t)$ induces a well defined quadratic form $\mathfrak{R}_{\lam(t)}:T_{\gamma(t)}M \times T_{\gamma(t)}M\to \R$
\begin{equation}
\Rcan_{\lam(t)}(v,v) := \sum_{i,j=1}^n R_{ij}(t) v_{i}v_j , \qquad v = \sum_{i=1}^n v_i X_i|_{\gamma(t)}.
\end{equation}
Indeed one can prove that 
\begin{equation}\label{eq:trc}
\Rcan_{\lam(t)}(v,v) = g(R^\nabla(v,\dot{\gamma})\dot{\gamma},v), \qquad v \in T_{\gamma(t)}M.
\end{equation}
The proof is a standard computation that can be found, for instance, in \cite[Appendix C]{BR-comparison}. Then, in the Jacobi equation \eqref{eq:Jacobicoord}, one has $C_1(t) =0$, $C_2(t) = \mathbb{I}$ (in particular, they are constant matrices), and the only non-trivial block $R(t)$ is the curvature operator along the geodesic: 
\begin{equation}
\dot x=p, \qquad \dot{p} = -R(t) x,
\end{equation}

\section{The sub-Riemannian case}
The problem of finding a the set of Darboux frames normalizing the Jacobi equation has been first studied by Agrachev-Zelenko in \cite{agzel1,agzel2} and subsequently completed by Zelenko-Li in \cite{lizel} in the general setting of curves in the Lagrange Grassmannian. A dramatic simplification, analogous to the Riemannian one, cannot be achieved in the general sub-Riemannian setting. Nevertheless, it is possible to find a normal form of \eqref{eq:Jacobicoord} where the matrices $C_{1}$ and $C_{2}$ are constant. Moreover, the very block structure of these matrices depends on the geodesic and already contains  important geometric invariants, that we now introduce.

\subsection{Geodesic flag and Young diagram}\label{s:gfyd}


Let $\gamma(t)$ be a  geodesic. Recall that  $\dot{\gamma}(t) \in \distr_{\g(t)}$ for every $t$. Consider a smooth admissible extension of the tangent vector, namely a vector field $\tanf \in \Gamma(\distr)$ such that $\tanf|_{\gamma(t)} = \dot{\gamma}(t)$.
\begin{definition}\label{d:flag}
The \emph{flag of the geodesic} $\gamma(t)$ is the sequence of subspaces
\begin{equation}
\DD_{\gamma(t)}^i :=  \spn\{\mc{L}_\tanf^j (X)|_{\gamma(t)} \mid  X \in \Gamma(\distr),\, j \leq i-1\} \subseteq T_{\gamma(t)} M, \qquad \forall\, i \geq 1,
\end{equation}
where $\mc{L}_{\tanf}$ denotes the Lie derivative in the direction of $\tanf$.
\end{definition}
By definition, this is a filtration of $T_{\gamma(t)}M$, i.e.~$\DD_{\gamma(t)}^i \subseteq \DD_{\gamma(t)}^{i+1}$, for all $i \geq 1$. Moreover, $\DD_{\gamma(t)}^1 = \distr_{\gamma(t)}$. Definition~\ref{d:flag} is well posed, namely does not depend on the choice of the admissible extension $\tanf$ (see \cite[Sec. 3.4]{curvature}).
The \emph{growth vector} of the geodesic $\gamma(t)$ is the sequence of integer numbers 
\begin{equation}
\mathcal{G}_{\gamma(t)} := \{\dim \DD_{\gamma(t)}^1,\dim \DD_{\gamma(t)}^2,\ldots\}.
\end{equation}
A geodesic $\gamma(t)$, with growth vector $\mathcal{G}_{\gamma(t)}$, is said
\begin{itemize}
\item \emph{equiregular} if $\dim \DD_{\gamma(t)}^i$ does not depend on $t$ for all $i \geq 1$,
\item \emph{ample} if for all $t$ there exists $m \geq 1$ such that $\dim \DD_{\gamma(t)}^{m} = \dim T_{\gamma(t)}M$.
\end{itemize}
Equiregular (resp.\ ample) geodesics are the microlocal counterpart of equiregular (resp.\ bracket-generating) distributions. Let $d_i:= \dim \DD_\gamma^i - \dim \DD_\gamma^{i-1}$, for $i\geq 1$, be the increment of dimension of the flag of the geodesic at each step (with the convention $\dim\mathcal{F}^0=0$).
\begin{lemma}[\cite{curvature}]\label{l:decreasing}
For an equiregular, ample geodesic, $d_1 \geq d_2 \geq \ldots \geq d_m$.
\end{lemma} 
 The generic geodesic is ample and equiregular. More precisely, the set of points $x \in M$ such that there a exists non-empty Zariski open set $A_{x} \subseteq T_{x}^*M$ of initial covectors for which the associated geodesic is ample and equiregular with the same (maximal) growth vector, is open and dense in $M$. See \cite{curvature,lizel} for more details.
 
For an ample, equiregular geodesic
we can build a tableau $\y$ with $m$ columns of length $d_{i}$, for $i=1,\ldots,m$, as follows:
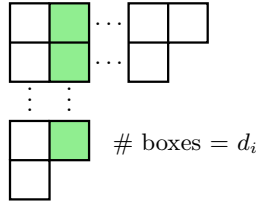
\begin{center}
\input{Yd0.tex}

\end{center}
Indeed $\sum_{i=1}^m d_i = n=\dim M$ is the total number of boxes in $\y$. 

Consider an ample, equiregular geodesic, with Young diagram $\y$, with $k$ rows, of length $n_1,\ldots,n_k$. Indeed $n_1+\ldots+n_k = n$. The moving frame we are going to introduce is indexed by the boxes of the Young diagram. The notation $ai \in \y$ denotes the generic box of the diagram, where $a=1,\ldots,k$ is the row index, and $i=1,\ldots,n_a$ is the progressive box number, starting from the left, in the specified row. We employ letters $a,b,c,\dots$ for rows, and $i,j,h,\dots$ for the position of the box in the row. 

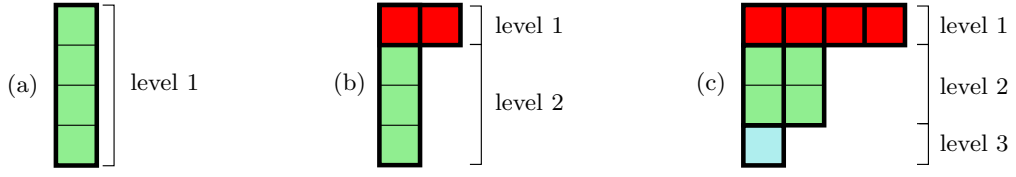
\begin{figure}[ht]
\centering
\input{superboxes.tex}
\caption{Levels (shaded regions) and superboxes (delimited by bold lines) for the Young diagram of (a) Riemannian, (b) contact, (c) a more general structure. The Young diagram for any Riemannian geodesic has a single level and a single superbox.}\label{f:Yd2}
\end{figure}
We collect the rows with the same length in $\y$, and we call them \emph{levels} of the Young diagram.  In particular, a level is the union of $r$ rows $\y_1,\ldots,\y_r$, and $r$ is called the \emph{size} of the level. The set of all the boxes $ai \in\y$ that belong to the same column and the same level of $\y$ is called \emph{superbox}. We use Greek letters $\alpha,\beta,\dots$ to denote superboxes. Notice that that two boxes $ai$, $bj$ are in the same superbox if and only if $ai$ and $bj$ are in the same column of $\y$ and in possibly distinct row but with same length, i.e.~if and only if $i=j$ and $n_a = n_b$ (see Fig.~\ref{f:Yd2}).

The following theorem is proved in \cite{lizel}.
\begin{theorem}\label{p:can} Assume $\lam(t)$ is the lift of an ample and equiregular geodesic $\g(t)$ with Young diagram $\y$. Then
there exists a smooth moving frame $\{E_{ai},F_{ai}\}_{ai \in \y}$ along $\lambda(t)$ such that
\begin{itemize}
\item[(i)] $\pi_{*}E_{ai}|_{\lambda(t)}=0$.
\item[(ii)] It is a Darboux basis, namely
\[
\sigma(E_{ai},E_{bj}) = \sigma(F_{ai},F_{bj}) = \sigma(E_{ai},F_{bj}) = \delta_{ab}\delta_{ij}, \qquad ai,bj \in \y.
\]
\item[(iii)] The frame satisfies structural equations
\begin{equation}\label{zelframe}
\displaystyle\begin{cases}	
\dot{E}_{ai} = E_{a(i-1)} & a = 1,\dots,k,\quad i = 2,\dots, n_a,\\[0.1cm]
\dot{E}_{a1} = -F_{a1} & a= 1,\dots,k, \\[0.1cm]
\dot{F}_{ai} = \sum_{bj \in \y} R_{ai,bj}(t) E_{bj} - F_{a(i+1)} & a=1,\dots,k,\quad i = 1,\dots,n_a-1,\\[0.1cm]
\dot{F}_{an_a} = \sum_{bj \in \y} R_{an_a,bj}(t) E_{bj}  & a = 1, \dots,k,
\end{cases}
\end{equation}
for some smooth family of $n\times n$ symmetric matrices $R(t)$, with components $R_{ai,bj}(t) = R_{bj,ai}(t)$, indexed by the boxes of the Young diagram $\y$. The matrix $R(t)$ is \emph{normal}  in the sense of \cite{lizel} (see Appendix~\ref{s:appendixnormal}). 
\end{itemize}
If $\{\wt{E}_{ai},\wt{F}_{ai}\}_{ai \in \y}$ is another smooth moving frame along $\lambda(t)$ satisfying (i)-(iii), with some normal matrix $\wt{R}(t)$, then for any superbox $\alpha$ of size $r$ there exists an orthogonal constant $r\times r$ matrix $O^\alpha$ such that
\begin{equation}
\wt{E}_{ai} = \sum_{bj \in \alpha} O^\alpha_{ai,bj} E_{bj}, \qquad \wt{F}_{ai} = \sum_{bj \in \alpha} O^\alpha_{ai,bj} F_{bj}, \qquad ai \in \alpha.
\end{equation}
\end{theorem}
\begin{remark}\label{rmk:notation}
For $a=1,\dots,k$, the symbol $E_a$ denotes the $n_a$-dimensional column vector
$
E_a = (E_{a1},E_{a2},\dots,E_{an_a})^*,
$
with  analogous notation for $F_a$. Similarly, $E$ denotes the $n$-dimensional column vector
$
E = (E_1,\dots,E_k)^*,
$
and similarly for $F$. Then, we rewrite the system \eqref{zelframe} as follows (compare with~\eqref{eq:Jacobiframe})
\begin{equation}\label{eq:Jacobiframe2}
\begin{pmatrix}
\dot{E} \\
\dot{F}
\end{pmatrix} = 
\begin{pmatrix}
C^*_1 & -C_2 \\
R(t) & -C_1
\end{pmatrix} \begin{pmatrix}
E\\
F
\end{pmatrix},
\end{equation}
where $C_1 = C_1(\y)$, $C_2=C_2(\y)$ are $n\times n$ matrices, depending on the Young diagram $\y$, defined as follows:
 for $a,b = 1,\dots,k$, $i=1,\dots,n_a$, $j=1,\dots,n_b$:
\begin{equation}
[C_1]_{ai,bj} := \delta_{ab}\delta_{i,j-1}, \label{eq:G1},\qquad
[C_2]_{ai,bj} := \delta_{ab}\delta_{i1}\delta_{j1}. 
\end{equation}
It is convenient to see $C_1$ and $C_2$ as block diagonal matrices:
\begin{equation}
C_i(\y) := \begin{pmatrix} 
C_i(\y_1) &  &    \\
 &  \ddots & \\
 &   & C_i(\y_k)
\end{pmatrix}, \qquad i =1,2,
\end{equation}
the $a$-th block being the $n_a\times n_a$ matrices
\begin{equation}\label{eq:Gamma}
C_1(\y_a) := \begin{pmatrix}
0 & \mathbb{I}_{n_a-1} \\
0 & 0
\end{pmatrix} , 
\qquad C_2(\y_a) := \begin{pmatrix}
1 & 0 \\
0 & 0_{n_a-1}
\end{pmatrix},
\end{equation}
where $\mathbb{I}_{m}$ is the $m \times m$ identity matrix and $0_{m}$ is the $m \times m$ zero matrix. Notice that the matrices $C_{1},C_{2}$ satisfy the Kalman rank condition
\begin{equation}\label{eq:Kalman}
\rank\{C_{2},C_{1}C_{2},\ldots,C_{1}^{n-1}C_{2}\}=n.
\end{equation}
Analogously, the matrices $C_{i}(D_{a})$ satisfy \eqref{eq:Kalman} with $n=n_{a}$.
\end{remark}

Let $\{X_{ai}\}_{ai \in \y}$ be the moving frame along $\gamma(t)$ defined by $X_{ai}|_{\gamma(t)}=\pi_{*}F_{ai}|_{\lam(t)}$, for some choice of a canonical Darboux frame. Theorem~\ref{p:can} implies that the following definitions are well posed.
\begin{definition}
The \emph{canonical splitting} of $T_{\gamma(t)} M$ is
\begin{equation}
T_{\gamma(t)}M = \bigoplus_{\alpha}S_{\gamma(t)}^{\alpha}, \qquad S_{\gamma(t)}^{\alpha}:=\spn\{ X_{ai}|_{\gamma(t)}\mid \, ai \in \alpha\},
\end{equation}
where the sum is over the superboxes $\alpha$ of $\y$. Notice that the dimension of $S_{\gamma(t)}^{\alpha}$ is equal to the size $r$ of the level to which the superbox $\alpha$ belongs.
\end{definition}

\begin{definition}\label{d:curv}
The \emph{canonical curvature} (along $\lambda(t)$), is the quadratic form  $\Rcan_{\lam(t)}: T_{\gamma(t)} M \times T_{\gamma(t)} M\to \R$ whose representative matrix, in terms of the basis $\{X_{ai}\}_{ai \in \y}$, is $R_{ai,bj}(t)$. In other words
\begin{equation}\label{eq:SR-dircurv}
\Rcan_{\lam(t)}(v,v) := \sum_{ai,bj\in\y} R_{ai,bj}(t) v_{ai}v_{bj}, \qquad v = \sum_{ai\in\y} v_{ai} X_{ai}|_{\gamma(t)} \in T_{\gamma(t)}M.
\end{equation}
We denote the restrictions of $\mathfrak{R}_{\lambda(t)}$ on the appropriate subspaces by:
\begin{equation}
\mathfrak{R}^{\alpha\beta}_{\lambda(t)} : S^{\alpha}_{\gamma(t)} \times S^{\beta}_{\gamma(t)} \to \R.
\end{equation}
For any superbox $\alpha$ of $D$, the \emph{canonical Ricci curvature} is the partial trace:
\begin{equation}
\mathfrak{Ric}_{\lambda(t)}^\alpha:= \sum_{ai \in \alpha} \Rcan_{\lambda(t)}^{\alpha\alpha}(X_{ai},X_{ai}).
\end{equation}
\end{definition}

The Jacobi equation, written in terms of the components $(p(t),x(t))$ with respect to a canonical Darboux frame $\{E_{ai},F_{ai}\}_{ai \in \y}$, becomes
\begin{equation}
\begin{pmatrix}\label{eq:Jacobicoord2}
\dot{p} \\ \dot{x}
\end{pmatrix} = \begin{pmatrix} - C_1 & -R(t) \\ C_2 & C_1^*
\end{pmatrix} \begin{pmatrix}
p \\ x
\end{pmatrix}.
\end{equation}
This is the sub-Riemannian generalization of the classical Jacobi equation seen as first-order equation for fields on the cotangent bundle. Its  structure depends on the Young diagram of the geodesic through the matrices $C_{i}(\y)$, while the remaining  invariants are contained in the curvature matrix $R(t)$. Notice that this includes the Riemannian case, where $\y$ is the same for every geodesic, with $C_{1}=0$ and $C_{2}=\mathbb{I}$.


\subsection{Homogeneity properties}

For all $c>0$, let $H_c := H^{-1}(c/2)$ be the Hamiltonian level set. In particular $H_1$ is the unit cotangent bundle: the set of initial covectors associated with unit-speed geodesics. Since the Hamiltonian function is fiber-wise quadratic, we have the following property for any $c>0$
\begin{equation}\label{eq:commutation}
e^{t \vec{H}}(c \lambda) = c e^{c t\vec{H}}(\lambda),
\end{equation}
where, for $\lambda \in T^*M$, the notation $c \lambda$ denotes the fiber-wise multiplication by $c$. Let $P_c : T^*M \to T^*M$ be the map $P_c(\lambda) = c\lambda$. Indeed $\alpha \mapsto P_{e^\alpha}$ is a one-parameter group of diffeomorphisms. Its generator is the \emph{Euler vector field} $\mathfrak{e} \in \Gamma(\ver)$, and is characterized by $P_c  = e^{(\ln c)\mathfrak{e}}$.
We can rewrite~\eqref{eq:commutation} as the following commutation rule for the flows of $\vec{H}$ and $\mathfrak{e}$:
\begin{equation}
e^{t\vec{H}} \circ P_c = P_c \circ e^{ c t \vec{H}}.
\end{equation}
Observe that $P_c$ maps $H_1$ diffeomorphically on $H_{c}$. Let $\lambda \in H_1$ be associated with an ample, equiregular geodesic with Young diagram $\y$. Clearly also the geodesic associated with $\lambda^c:=c \lambda \in H_{c}$ is ample and equiregular, with the same Young diagram. This corresponds  to a reparametrization of the same curve: in fact $\lambda^c(t) =e^{t\vec{H}}(c\lam)= c(\lambda(ct))$, hence $\gamma^c(t) = \pi(\lambda^c(t))= \gamma(ct)$.

\begin{theorem}[Homogeneity properties of the canonical curvature]\label{t:homogR}
For any superbox $\alpha \in \y$, let $|\alpha|$ denote the column index of $\alpha$. Denoting $\lam^{c}(t)=e^{t\vec{H}}(c\lam)$ we have, for any $c>0$
\begin{equation}
\mathfrak{R}^{\alpha\beta}_{\lambda^{c}(t)} = c^{|\alpha|+ |\beta|} \mathfrak{R}^{\alpha\beta}_{\lambda(ct)},
\end{equation}
\end{theorem}
\begin{remark}
In the Riemannian setting, $\y$ has only one superbox with $|\al|=1$  (see Fig.~\ref{f:Yd2}). Then $\mathfrak{R}_{\lambda}:=\mathfrak{R}^{\alpha\alpha}_{\lambda(0)}$ is homogeneous of degree $2$ as a function of $\lam$. 
\end{remark}
Theorem \ref{t:homogR} follows  directly from the next result and Definition \ref{d:curv}. In the next proposition, for any $\eta \in T^*M$ and $c >0$, we denote with $d_{\eta} P_c : T_{\eta}(T^*M) \to T_{c\eta}(T^*M)$ the differential of the map $P_c$, computed at $\eta$.

\begin{proposition}\label{p:framescaling}
Let $\lambda \in H_1$ and $\{E_{ai},F_{ai}\}_{ai \in \y}$ be the associated canonical frame along the extremal $\lambda(t)$. Let $c>0$ and define, for $ai \in \y$
\begin{equation}
E^c_{ai}(t):=\frac{1}{c^i}(d_{\lambda(ct)} P_{c}) E_{ai}(ct), \qquad F_{ai}^c(t):=c^{i-1}(d_{\lambda(ct)} P_c) F_{ai}(ct).
\end{equation}
The moving frame $\{E^c_{ai}(t),F^c_{ai}(t)\}_{ai \in \y} \in T_{\lambda^c(t)}(T^*M)$ is a canonical frame associated with the initial covector $\lambda^c =c\lambda\in H_{c}$, with curvature matrix
\begin{equation}\label{eq:homog}
R^{\lambda^{c}}_{ai,bj}(t) = c^{i+j} R^\lambda_{ai,bj}(ct).
\end{equation}
\end{proposition}

\begin{proof}
We check all the relations of Theorem~\ref{p:can}. Indeed $P_\alpha$ sends fibers to fibers, hence (i) is trivially satisfied. For what concerns (ii), let $\theta$ be the Liouville one-form, and $\sigma =d\theta$. Indeed $P_c^* \theta = c\theta$. Hence $P_c^*\sigma = c\sigma$. It follows that $\{E^c_{ai}(t),F^c_{ai}(t)\}_{ai \in \y}$ is a Darboux frame at $\lambda^c(t)$:
\begin{equation}
\sigma_{\lambda^c(t)}(E^c_{ai}(t),F^c_{bj}(t)) = \tfrac{1}{c} (P_c^* \sigma)_{\lambda(t)}(E_{ai}(t),F_{bj}(t)) = \delta_{ab}\delta_{ij},
\end{equation}
and similarly for the others Darboux relations.

For what concerns (iii) (the structural equations), let $\xi(t)$ be any vector field along $\lambda(t)$, and $(d_{\lambda(t)} P_c) \xi(ct)$ be the corresponding vector field along $\lambda^c(t)$. Then
\begin{align*}
\left.\frac{d}{d\eps}\right|_{\eps =0} e^{-\eps \vec{H}}_* \circ (d_{\lambda(t)} P_c) \xi (c(t+\eps)) & =\left.\frac{d}{d\eps}\right|_{\eps =0} (e^{-\eps \vec{H}} \circ P_c)_* \xi(c(t+\eps)) \\
& =\left.\frac{d}{d\eps}\right|_{\eps =0} (P_c \circ e^{-c\eps \vec{H}})_* \xi(c(t+\eps)) \\
& =\left.c\frac{d}{d\tau}\right|_{\tau =0} (P_c \circ e^{-\tau \vec{H}})_* \xi(ct+\tau) \\
& = c (d_{\lambda(ct)}P_c) \dot\xi(ct).
\end{align*}
Applying the above identity to compute the derivatives of the new frame, and using~\eqref{zelframe}, one finds that $\{E^c_{ai}(t),F^c_{ai}(t)\}_{ai \in \y}$ satisfies the structural equations, with curvature matrix given by~\eqref{eq:homog}. For example
\begin{align*}
\dot{F}_{ai}^c(t) & = c^{i-1} c (d_{\lambda(ct)} P_c) \dot{F}_{ai}(ct) \\
& = c^i (d_{\lambda(ct)}P_c) [R^{\lam}_{ai,bj}(ct) E_{bj}(ct) - F_{a(i+1)}(ct)] \\
& = c^i [c^j R^{\lam}_{ai,bj}(ct) E_{bj}^c(t) - c^{-i} F^c_{a(i+1)}(t)] \\
& = c^{i+j}R^{\lam}_{ai,bj}(ct) E_{bj}^c(t) - F^c_{a(i+1)}(t),
\end{align*}
where we suppressed a summation over $bj \in \y$.
\end{proof}

 Proposition \ref{p:can} defines not only a curvature, but also a (non-linear) connection, in the sense of Ehresmann, that we now introduce.

\section{Ehresmann curvature and curvature operator}
For any smooth vector bundle $N$ over $M$, let $\Gamma(N)$ denote the smooth sections of $N$. 
Recall that $\ver :=\ker \pi_* \subset T(T^*M)$ is the \emph{vertical distribution}. An \emph{Ehresmann connection} on $T^{*}M$ is a smooth distribution $\hor \subset T (T^*M)$   such that
\begin{equation}
T(T^*M) = \hor \oplus \ver.
\end{equation}
We call $\hor$ the \emph{horizontal distribution}\footnote{ Note that this is a distribution on $T^{*}M$, i.e.~a sub-bundle of $T(T^*M)$ and should not be confused with the sub-Riemannian distribution $\distr$, that is a subbundle of $TM$.}. An Ehresmann connection $\hor$ is \emph{linear} if $\hor_{c\lam}=(d_{\lam}P_{c} ) \hor_{\lam}$ for every $\lam\in T^{*}M$ and $c>0$.

For any $X \in \Gamma(TM)$ there exists a unique \emph{horizontal lift} $\nabla_X $ in $ \Gamma(\hor)$ such that $\pi_* \nabla_X = X$.  
\begin{remark}
A function $h \in C^\infty(T^*M)$ is fiber-wise linear if it can be written as $h(\lam)=\la \lam,Y\ra$, for some $Y\in \Gamma(TM)$. Such an $Y$ is clearly unique, and for this reason we denote $h_Y := \lam \mapsto \la \lam,Y\ra$ the fiber-wise linear function associated with $Y \in \Gamma(TM)$.
 A connection $\nabla$ is linear if, for every $X\in \Gamma(TM)$, the derivation $\nabla_{X}$ maps  fiber-wise linear functions to fiber-wise linear functions. In this case, we recover the classical notion of covariant derivative by defining $\nabla_{X}Y=Z$ if $\nabla_{X}h_{Y}=h_{Z}$, where $Y,Z\in \Gamma(TM)$.
\end{remark}

 We recall the definition of curvature of an Ehresmann connection \cite{KNFDG}.
\begin{definition}
The \emph{ Ehresmann curvature} of the connection $\nabla$ is the $C^\infty(M)$-linear map $R^{\nabla}: \Gamma(TM) \times \Gamma(TM) \to \Gamma(\ver)$ defined by
\begin{equation}
R^{\nabla}(X,Y) = [\nabla_X,\nabla_Y] - \nabla_{[X,Y]}, \qquad X,Y \in \Gamma(TM).
\end{equation}
\end{definition}
$R^{\nabla}$ is skew-symmetric, namely $R^{\nabla}(X,Y) = -R^{\nabla}(Y,X)$. Notice that $R^{\nabla}=0$ if and only if $\hor$ is involutive. 
\subsection{Canonical connection}
Let $\g(t)$ be a fixed ample and equiregular geodesic with Young diagram $\y$, projection of the extremal $\lam(t)$, with initial covector $\lam$. Let $\{E_{ai}(t),F_{ai}(t)\}$ be a canonical frame along $\lambda(t)$. For $t=0$, this defines a subspace at $\lam\in T^*M$, namely
\begin{equation}\label{eq:cancon}
\hor_\lambda := \spn\{F_{ai}|_{\lambda}\}_{ai \in \y}, \qquad \lambda \in T^*M .
\end{equation}
Indeed this definition makes sense on the subset of covectors $N\subset T^{*}M$ associated with ample and equiregular geodesics. In the Riemannian case, every non-trivial geodesic is ample and equiregular, with the same Young diagram. Hence $N=T^{*}M\setminus H^{-1}(0)$. A posteriori one can show that this connection is linear and  can be extended smoothly  on the whole $T^*M$. In the sub-Riemannian case, $N\subset T^{*}M\setminus H^{-1}(0)$.

In general, using the results of \cite[Section 5.2]{curvature} and \cite[Section 5]{lizel}, one can prove that $N$ is open and dense in $T^{*}M$. Moreover, the elements of the frame depend rationally (in charts) on the point $\lambda$, hence $\hor$ is smooth on $N$. 

For simplicity, we assume that it is possible to extend $\hor$ to a smooth  distribution on the whole $T^*M$. This is indeed possible in some cases of interest: on corank $1$ structures with symmetries \cite{lizel2} and on contact sub-Riemannian structures \cite{nostrocontact} (see also \cite{RS-3-Sasakian} for fat structures). In the general case, we replace $T^*M$ with $N$. 

\begin{definition} The \emph{canonical Ehresmann connection} associated with the sub-Riemannian structure is the horizontal distribution $\hor \subset T (T^{*}M)$ defined by \eqref{eq:cancon}.
\end{definition}

As a consequence of Proposition \ref{p:framescaling}, $\hor$ is non-linear, in general.  However, if the structure is Riemannian, one has $\hor_{c\lam}=(d_{\lam}P_{c} ) \hor_{\lam}$
and the connection is linear.

\begin{proposition}Let $H$ be the sub-Riemannian Hamitonian and $\hor$ the canonical connection. Then $\nabla_{X} H=0$ for every $X\in\Gamma(TM)$. Equivalently, $\vec{H}\in \hor$.
\end{proposition}
\begin{remark} The above condition is the compatibility of the canonical connection with the sub-Riemannian metric. In the Riemannian setting,  $\hor$ is linear and this condition can be rewritten, in the sense of covariant derivative, as $\nabla g=0$.
\end{remark}
\begin{proof} The equivalence of the two statements follows from the definition of Hamiltonian vector field and the fact that  $\hor$ is Lagrangian, by construction. Indeed 
\begin{equation}
\nabla_{X}H=dH(\nabla_{X})=\sigma(\vec{H},\nabla_{X}).
\end{equation}
Then we prove that $\vec H\in \hor$.
\begin{lemma}\label{l:dere}
Let $\mathfrak{e}$ be the Euler vector field. Then $\dot{\mathfrak{e}} = -\vec{H}$.
\end{lemma}
\begin{proof}[Proof of Lemma \ref{l:dere}]
Let $P_s = e^{(\ln s) \mathfrak{e}}$ be the dilation along the fibers. We have the following commutation rule for the flows of $\vec{H}$ and $\mathfrak{e}$
\begin{equation}
P_{-s} \circ e^{-t\vec{H}} \circ P_s = e^{ -ts\vec{H}}.
\end{equation}
Computing the derivative w.r.t $t$ and $s$ at $(t,s) = (0,1)$ we obtain $[\vec{H},\mathfrak{e}] = - \mathfrak{e}$, that implies the statement.
\end{proof}
\begin{lemma}\label{l:mform0}
 Since $\mathfrak{e}$ is vertical, then $\mathfrak{e} = v(t)^* E(t)$ for some smooth $v(t) \in \R^n$. Accordingly with the decomposition of Remark~\ref{rmk:notation}, we set
\begin{equation}
v(t) = (v_1(t),\ldots,v_k(t))^{*}, \quad \text{with} \quad v_a(t) =  (v_{a1}(t),\ldots,v_{an_a}(t))^{*}.
\end{equation}
Then $v(t)$ is constant and we have
\begin{equation}
\mathfrak{e}=\sum_{\substack{ai\in \y \\ n_{a}=1}}    v_{ai} E_{ai}.
\end{equation}
\end{lemma}

\begin{proof}[Proof of Lemma~\ref{l:mform0}]
As a consequence of Lemma~\ref{l:dere}, $\ddot{\mathfrak{e}} = 0$. Using the structural equations~\eqref{eq:Jacobiframe2}, we obtain
\begin{align}
C_1^* C_2v - C_2 C_1 v - 2 C_2 \dot{v} & = 0,  \label{eq:Fpart}\\
\ddot{v} + 2 C_1 \dot{v} + C_1^2 v  - RC_2 v  &= 0. \label{eq:Epart}
\end{align}
We show that for any row index of the Young diagram $a=1,\ldots,k$
\begin{equation}
v_a = \begin{cases}
(0,\ldots,0)^{*} & n_a > 1, \\
\text{constant} & n_a = 1.
\end{cases}
\end{equation}

Let us focus on~\eqref{eq:Fpart}. For each $a=1,\ldots,k,$ we take its $a$-th block. By the block structure of $C_1$ and $C_2$, this is
\begin{equation}\label{eq:Fpart2}
C_1^* C_2 v_a - C_2 C_1 v_a - 2 C_2 \dot{v}_a  = 0, \qquad  \forall\, a=1,\ldots,k,
\end{equation}
where here $C_1 = C_1(\y_a)$ and $C_2 = C_2(\y_a)$. If $n_a = 1$, then $C_1 = 0$ and $C_2 = 1$. In this case~\eqref{eq:Fpart2} implies $v_a(t) = v_a$ is constant. Now let $n_a > 1$. In this case, the particular form of $C_1,C_2$ for~\eqref{eq:Fpart2} yields
\begin{equation}
C_1^* C_2 v_a = 0, \qquad \text{and} \qquad C_2 C_1 v_a + 2C_2 \dot{v}_a = 0, \qquad (n_a > 1).
\end{equation}
Indeed the kernel of $C_1^*$ is orthogonal to the image of $C_2$. Hence $C_1^* C_2 v_a=0$ implies $C_2 v_a = 0$. In particular~\eqref{eq:Fpart2} is equivalent to
\begin{equation}\label{eq:Fpart3}
C_2 v_a = 0,  \qquad C_2 C_1 v_a = 0, \qquad (n_a > 1).
\end{equation}
More explicitly, $v_a = (0,0,v_{a3},\ldots,v_{an_a})$. For the case $n_a =2$ this is sufficient to completely determine $v_a$. In all the other cases, let us turn to~\eqref{eq:Epart}. The latter does not split immediately, as the curvature matrix $R$ is not block-diagonal. However, let us consider a copy of~\eqref{eq:Epart} multiplied by $C_2 C_1^i$. For each $a$ such that $n_{a}>2$ we consider its $a$-th block, obtaining the following:
\begin{equation}\label{eq:Epart2}
C_2 C_1^i \ddot{v}_a + 2 C_2 C^{i+1}_1 \dot{v}_a + C_2 C_1^{i+2} v_a  - [C_2 C_1^i RC_2 v]_a  = 0, \qquad (n_a > 2).
\end{equation}
We claim that $[C_2 C_1^i RC_2 v]_a = 0$ if $n_a > 2$ and $i < n_a -2$. 

By setting the matrix $[R_{ab}]_{ij} := R_{ai,bj}$, with $ai,bj \in \y$ (this is a block of $R$, corresponding to the rows $a,b$ of the Young diagram $\y$), we compute
\begin{align}
[C_2 C_1^i R C_2 v]_a & = \sum_{b,c,d=1}^k [C_2 C_1^i]_{ab} R_{bc} [C_2]_{cd} v_d = \sum_{b=1}^k (C_2 C_1^i) R_{ab} (C_2 v_b) \\
& = \sum_{n_b =1} (C_2 C_1^i) R_{ab} (C_2 v_b)  = \sum_{n_b =1} R_{a(i+1),b1} v_{b1},
\end{align}
where we used the block structure of the $C_i$'s and \eqref{eq:Fpart3}. The last sum involves only $R_{a(i+1),b1}$ with $n_b =1$ and $n_a > 2$. If $i< n_a -2$, then $R_{a(i+1),b1}$ is \emph{not} in the last $2n_b = 2$ elements of Table~\ref{t:normaltable}, and vanishes by the normal conditions (see Appendix \ref{s:appendixnormal}).
Thus we have:
\begin{equation}\label{eq:Epart3}
C_2 C_1^i \ddot{v}_a + 2 C_2 C^{i+1}_1 \dot{v}_a + C_2 C_1^{i+2} v_a=0 , \qquad (n_a > 2, \quad i < n_a -2).
\end{equation}
In particular using~\eqref{eq:Fpart3}, and taking $i=0,\ldots,n_a-3$ we see that~\eqref{eq:Epart3} is equivalent to $C_2 C_1^{i+2} v_a = 0$ for all $i=0,\ldots,n_{a}-3$. Combining all the cases
\begin{equation}
v_a \in \ker\{ C_2, C_2 C_1, C_2 C_1^2,\ldots,C_2 C_1^{n_a -1}\}, \qquad (n_a > 1).
\end{equation}
This yields $v_a = 0,$ by Kalman rank condition~\eqref{eq:Kalman}. 
\end{proof}
Lemma \ref{l:mform0} implies our statement since 
\begin{equation}
\vec{H}=-\dot{\mathfrak{e}}=-\sum_{\substack{ai\in \y \\ n_{a}=1}}    v_{ai} \dot E_{ai}=\sum_{\substack{ai\in \y \\ n_{a}=1}}    v_{ai}  F_{ai}\in \hor,
\end{equation}
where we used the structural equations \eqref{zelframe} for the $E_{ai}$'s with $n_{a}=1$.
\end{proof}

\subsection{Relation with the canonical curvature}
We now discuss the relation between the curvature of the canonical Ehresmann connection and  the sub-Riemannian curvature operator. In what follows we denote by $\Rcan_{\lambda}:=\Rcan_{\lam(0)}$, where $\lam(t)$ is the extremal with initial datum $\lam$.
Then $\Rcan$ extends to a well defined map
\begin{equation} \label{eq:cancurva}
\begin{gathered}
\Rcan : \Gamma(T^*M) \times \Gamma(TM) \times \Gamma(TM) \to C^\infty(M), \\ 
(\lambda, X, Y) \mapsto \Rcan_\lambda(X,Y).
\end{gathered}
\end{equation}
We stress that here the first argument is a section $\lam\in\Gamma(T^*M)$. 

Although $\Rcan$ is $C^\infty(M)$-linear in the last two arguments by construction, it is in general non-linear in the first argument, so it does not define a $(1,2)$ tensor. Nevertheless, for any fixed section $\lambda \in \Gamma(T^*M)$, the restriction $\Rcan_\lambda : \Gamma(TM) \times \Gamma(TM) \to C^\infty(M)$ is a $(0,2)$ symmetric tensor.

\begin{theorem} \label{t:canehr}
Let $R^{\nabla} : \Gamma(TM) \times \Gamma(TM) \to \Gamma(\ver)$ be the  curvature of the canonical Ehresmann connection, and let $\Rcan: \Gamma(T^*M) \times \Gamma(TM) \times \Gamma(TM) \to C^\infty(M)$ be the canonical curvature map \eqref{eq:cancurva}. Then
\begin{equation}\label{eq:relation}
\Rcan_\lambda(X,Y) = \sigma_\lambda(R^{\nabla}(\tanf,X),\nabla_Y), \qquad  \forall\,\lambda \in \Gamma(T^*M), \quad X,Y \in \Gamma(TM),
\end{equation}
where $\tanf = \pi_*\vec{H}|_\lambda \in \Gamma(TM)$.
\end{theorem}
\begin{proof}
We evaluate the right hand side of \eqref{eq:relation} at the point $x$, for any fixed section $\lambda=\lam(x) \in \Gamma(T^*M)$. By linearity, it is sufficient to take $X = X_{ai}$ and $Y=Y_{bj}$, projections of a canonical frame $F_{ai}|_{\lam},F_{bj}|_{\lam}$ at $t=0$. Indeed, by definition, $ \nabla_{X_{ai}}|_{\lam} = F_{ai}|_{\lam}$. Then
\begin{align*}
\sigma_\lambda(R^\nabla(\tanf,X_{ai}),\nabla_{X_{bj}}) & = \sigma_\lambda([\nabla_{\tanf},F_{ai}],F_{bj}) 
 = \sigma_\lambda([\vec{H},F_{ai}],F_{bj}) \\
& = \sigma_\lambda(\dot{F}_{ai}, F_{bj}) 
= R^\lambda_{ai,bj}(0).
\end{align*}
Here we used the structural equations and that $\vec{H} \in \hor$, thus $\nabla_{\tanf} = \vec{H}$. By definition of canonical curvature map, we obtain the statement.
\end{proof}
\begin{remark}
For $\lambda \in \Gamma(T^*M)$, the corresponding tangent field $\tanf \in \Gamma(\distr) \subsetneq \Gamma(TM)$. Therefore, $\Rcan$ recovers only part of the whole Ehresmann connection.
\end{remark}
\begin{remark}[On the Riemannian case] As we proved in \eqref{eq:trc}, we have 
\[
\Rcan_\lambda(X,Y) =R^\nabla(\tanf,X,Y,\tanf),
\]
 where $\tanf = \pi_*\vec{H}|_\lambda$ is the tangent vector associated with the covector $\lambda$. For  completeness, let us recover the same formula by the r.h.s. of \eqref{eq:relation}. Indeed, for any vertical vector $V \in \ver_\lambda$ and $W \in T_{\lambda}(T^*M)$, we have $\sigma_{\lambda}(V,W) = V(h_{\pi_*W})|_{\lambda}$ as one can check from a direct computation. Thus the r.h.s. of \eqref{eq:relation}  is
\begin{align*}
 \sigma_\lambda([\nabla_\tanf,\nabla_X]-\nabla_{[\tanf,X]},\nabla_Y)  
& = \left(\nabla_\tanf \nabla_X (h_Y) - \nabla_X \nabla_\tanf(h_Y) - \nabla_{[\tanf,X]} (h_Y)\right)|_{\lambda} \\ 
& = h_{\nabla_{\tanf}\nabla_X Y - \nabla_X \nabla_\tanf Y - \nabla_{[\tanf,X]} Y}(\lambda) \\
& = \langle \lambda, \nabla_{\tanf}\nabla_X Y - \nabla_X \nabla_\tanf Y - \nabla_{[\tanf,X]} Y \rangle \\
& = g(\nabla_{\tanf}\nabla_X Y - \nabla_X \nabla_\tanf Y - \nabla_{[\tanf,X]} Y, \tanf) \\
& = R^\nabla(\tanf,X,Y,\tanf).
\end{align*}
\end{remark}

\appendix
\section{Normal condition for the canonical frame}\label{s:appendixnormal}
Here we rewrite the \emph{normal} condition for the matrix $R_{ai,bj}$ mentioned in Theorem \ref{p:can} (and defined in \cite{lizel}) according to our notation.

\begin{definition} The matrix $R_{ai,bj}$ is \emph{normal} if it satisfies:
\begin{itemize}
\item[(i)] global symmetry: for all $ai,bj\in \y$
\[R_{ai,bj}=R_{bj,ai}.\]
\item[(ii)] partial skew-symmetry: for all $ai,bi\in \y$ with $n_{a}=n_{b}$ and $i<n_{a}$ 
\[R_{ai,b(i+1)}=- R_{bi,a(i+1)}.\]
\item[(iii)] vanishing conditions: the only possibly non vanishing entries $R_{ai,bj}$ satisfy
\begin{itemize}
\item[(iii.a)] $n_{a}=n_{b}$ and $|i-j|\leq 1$,
\item[(iii.b)] $n_{a}>n_{b}$ and $(i,j)$ belong to the last $2n_{b}$ elements of Table~\ref{t:normaltable}.
\begin{table}[htbp]
\begin{center}
\caption{Vanishing conditions.}
\begin{tabular}{|c||c|c|c|c|c|c|c|c|c|c|c|c|c|}
\hline
$i$ & $1$ & $1$ & $2$ & $\cdots$ & $\ell$ & $\ell$ & $\ell+1$ & $\cdots$ & $n_{b}$ & $n_{b}+1$ & $\cdots$ & $n_a-1$ & $n_a$ \\
\hline
$j$ & $1$ & $2$ & $2$ & $\cdots$ & $\ell$ & $\ell+1$ & $\ell+1$ & $\cdots$ & $n_b$ & $n_b$ & $\cdots$ & $n_b$ & $n_b$ \\
\hline
\end{tabular}
\label{t:normaltable}
\end{center}
\end{table}
\end{itemize}
\end{itemize}
The sequence is obtained as follows: starting from $(i,j)=(1,1)$ (the first boxes of the rows $a$ and $b$), each next even pair is obtained from the previous one by increasing $j$ by one (keeping $i$ fixed). Each next odd pair is obtained from the previous one by increasing $i$ by one (keeping $j$ fixed). This stops when $j$ reaches its maximum, that is $(i,j) = (n_b,n_b)$. Then, each next pair is obtained from the previous one by increasing $i$ by one (keeping $j$ fixed), up to $(i,j) = (n_a,n_b)$. The total number of pairs appearing in the table is $n_b+n_a-1$.
\end{definition}

\addtocontents{toc}{\protect\setcounter{tocdepth}{0}}
\section*{Acknowledgments} 
This research has  been supported by the European Research Council, ERC StG 2009 ``GeCoMethods'', contract number 239748, by the iCODE institute (research project of the Idex Paris-Saclay), and by the Grant ANR-15-CE40-0018 of the ANR. This research benefited from the support of the ``FMJH Program Gaspard Monge in optimization and operation research'' and from the support to this program from EDF.

\bibliographystyle{abbrv}
\bibliography{Connection-Biblio}
\end{document}

%% file: Yd0.tex
\begin{tikzpicture}[x=0.26mm, y=0.26mm, inner xsep=0pt, inner ysep=0pt, outer xsep=0pt, outer ysep=0pt]
\path[line width=0mm] (61.05,70.00) rectangle +(127.98,100.00);
\draw(120.00,159.00) node[anchor=base]{\fontsize{9.39}{11.27}\selectfont $\ldots$};
\draw(120.00,139.00) node[anchor=base]{\fontsize{9.39}{11.27}\selectfont $\ldots$};
\draw(80.00,115.00) node[anchor=base]{\fontsize{9.39}{11.27}\selectfont $\vdots$};
\draw(100.00,115.00) node[anchor=base]{\fontsize{9.39}{11.27}\selectfont $\vdots$};
\definecolor{L}{rgb}{0,0,0}
\path[line width=0.30mm, draw=L] (130.00,170.00) -- (130.00,130.00) -- (150.00,130.00) -- (150.00,150.00) -- (170.00,150.00) -- (170.00,170.00) -- cycle;
\path[line width=0.30mm, draw=L] (150.00,170.00) -- (150.00,150.00);
\path[line width=0.30mm, draw=L] (130.00,150.00) -- (150.00,150.00);
\draw(122.00,95.00) node[anchor=base west]{\fontsize{9}{10.24}\selectfont \# boxes = $d_i$};
\definecolor{F}{rgb}{0.565,0.933,0.565}
\path[line width=0.30mm, draw=L, fill=F] (90.00,170.00) [rotate around={270:(90.00,170.00)}] rectangle +(40.00,20.00);
\path[line width=0.30mm, draw=L] (90.00,170.00) -- (70.00,170.00) -- (70.00,130.00) -- (90.00,130.00);
\path[line width=0.30mm, draw=L] (70.00,150.00) -- (110.00,150.00);
\path[line width=0.30mm, draw=L, fill=F] (90.00,110.00) [rotate around={270:(90.00,110.00)}] rectangle +(20.00,20.00);
\path[line width=0.30mm, draw=L] (90.00,70.00) [rotate around={90:(90.00,70.00)}] rectangle +(20.00,20.00);
\path[line width=0.30mm, draw=L] (70.00,90.00) -- (70.00,110.00) -- (90.00,110.00);
\end{tikzpicture}%

%% file: superboxes.tex
\begin{tikzpicture}[x=0.30mm, y=0.30mm, inner xsep=0pt, inner ysep=0pt, outer xsep=0pt, outer ysep=0pt]
\path[line width=0mm] (-75.00,91.40) rectangle +(300.63,71.04);
\definecolor{L}{rgb}{0,0,0}
\definecolor{F}{rgb}{0.565,0.933,0.565}
\path[line width=0.60mm, draw=L, fill=F] (-53.50,91.50) rectangle +(17.92,70.94);
\path[line width=0.15mm, draw=L] (-53.30,127.00) -- (-35.58,127.00);
\path[line width=0.15mm, draw=L] (-53.30,109.29) -- (-35.58,109.29);
\path[line width=0.15mm, draw=L] (-53.30,91.57) -- (-35.58,91.57);
\path[line width=0.15mm, draw=L] (-53.30,144.72) -- (-35.58,144.72);
\path[line width=0.60mm, draw=L, fill=F] (88.32,91.67) rectangle +(17.82,70.76);
\definecolor{F}{rgb}{1,0,0}
\path[line width=0.60mm, draw=L, fill=F] (88.43,144.72) rectangle +(35.43,17.72);
\path[line width=0.60mm, draw=L] (106.14,162.44) -- (106.14,144.72);
\path[line width=0.15mm, draw=L] (88.43,127.00) -- (106.14,127.00);
\path[line width=0.15mm, draw=L] (88.43,109.29) -- (106.14,109.29);
\path[line width=0.15mm, draw=L] (88.43,91.57) -- (106.14,91.57);
\definecolor{F}{rgb}{0.686,0.933,0.933}
\path[line width=0.60mm, draw=L, fill=F] (247.87,91.55) [rotate around={0:(247.87,91.55)}] rectangle +(17.71,17.73);
\definecolor{F}{rgb}{0.565,0.933,0.565}
\path[line width=0.60mm, draw=L, fill=F] (247.87,109.29) rectangle +(35.43,35.43);
\definecolor{F}{rgb}{1,0,0}
\path[line width=0.60mm, draw=L, fill=F] (247.87,144.72) rectangle +(70.86,17.72);
\path[line width=0.60mm, draw=L] (265.58,162.44) -- (265.58,109.29);
\path[line width=0.15mm, draw=L] (247.87,127.00) -- (283.30,127.00);
\path[line width=0.15mm, draw=L] (247.87,91.57) -- (265.58,91.57);
\path[line width=0.60mm, draw=L] (283.30,162.44) -- (283.30,144.72);
\path[line width=0.60mm, draw=L] (301.01,162.44) -- (301.01,144.72);
\draw(-20.85,124.71) node[anchor=base west]{\fontsize{9.39}{11.27}\selectfont level 1};
\path[line width=0.15mm, draw=L] (128.00,162.00) -- (133.00,162.00) -- (133.00,145.00) -- (128.00,145.00);
\path[line width=0.15mm, draw=L] (128.00,145.00) -- (133.00,145.00) -- (133.00,92.00) -- (128.00,92.00);
\draw(139.00,150.00) node[anchor=base west]{\fontsize{9.39}{11.27}\selectfont level 1};
\draw(139.00,116.00) node[anchor=base west]{\fontsize{9.39}{11.27}\selectfont level 2};
\path[line width=0.15mm, draw=L] (324.00,162.00) -- (329.00,162.00) -- (329.00,145.00) -- (324.00,145.00);
\draw(334.00,150.00) node[anchor=base west]{\fontsize{9.39}{11.27}\selectfont level 1};
\path[line width=0.15mm, draw=L] (324.00,145.00) -- (329.00,145.00) -- (329.00,110.00) -- (324.00,110.00);
\draw(334.00,123.00) node[anchor=base west]{\fontsize{9.39}{11.27}\selectfont level 2};
\path[line width=0.15mm, draw=L] (325.00,110.00) -- (329.00,110.00) -- (329.00,92.00) -- (324.00,92.00);
\draw(334.00,98.00) node[anchor=base west]{\fontsize{9.39}{11.27}\selectfont level 3};
\draw(68.00,125.00) node[anchor=base west]{\fontsize{9}{10.24}\selectfont (b)};
\draw(227.00,125.00) node[anchor=base west]{\fontsize{9}{10.24}\selectfont (c)};
\draw(-75.00,124.00) node[anchor=base west]{\fontsize{9}{10.24}\selectfont (a)};
\path[line width=0.15mm, draw=L] (-33.00,162.40) -- (-28.00,162.40) -- (-28.00,91.40) -- (-33.00,91.40);
\end{tikzpicture}%